\documentclass[12pt]{article}

\usepackage[english]{babel}

\usepackage{amsmath}
\usepackage{graphicx}
\usepackage[colorlinks=true, allcolors=blue]{hyperref}
\usepackage{amssymb}
\usepackage{amsthm}
\usepackage{ dsfont } 
\usepackage{mathtools} 

\title{Critical Multitype Branching Processes with Random Migration}
\author{}
\date{}

\newcommand{\R}{\mathbb{R}}
\newcommand{\Z}{\mathbb{Z}}
\newcommand{\N}{\mathbb{N}}
\newcommand{\bZ}{\boldsymbol{Z}}
\newcommand{\bz}{\boldsymbol{z}}
\newcommand{\ev}[1]{\operatorname{E}\! \left[ #1 \right]}

\newcommand{\OO}{\operatorname{O}}
\newcommand{\oo}{\operatorname{o}}
\newcommand{\dif}{\,\mathrm{d}}
\newcommand{\var}[1]{\operatorname{Var}\! \left[ #1 \right]}

\newcommand{\evcond}[2]{\operatorname{E} \left[ #1 \;\middle|\; #2 \right]} 

\newtheorem{theorem}{Theorem}
\newtheorem{proposition}{Proposition}
\newtheorem{lemma}{Lemma}
\theoremstyle{definition}

\newtheorem*{remark}{Remark}
\newcommand{\proofend}{\hfill\mbox{$\Box$}}

\author{Miguel Gonz\'alez\ \footnote{Departamento de Matemáticas, Facultad de Ciencias and Instituto de Computaci\'on Cient\'ifica Avanzada, Universidad de Extremadura, Badajoz, Spain. e-mail address: \url{mvelasco@unex.es}. ORCID: \href{https://orcid.org/0000-0001-7481-6561}{0000-0001-7481-6561}.}
\and Pedro Mart\'in-Ch\'avez \footnote{Departamento de Matemáticas, Facultad de Ciencias, Universidad de Extremadura, Badajoz, Spain.  e-mail address: \url{pedromc@unex.es}. ORCID: \href{https://orcid.org/0000-0001-5530-3138}{0000-0001-5530-3138}.}
\and In\'es del Puerto\  \footnote{Departamento de Matemáticas, Facultad de Ciencias and Instituto de Computaci\'on Cient\'ifica Avanzada, Universidad de Extremadura,  Badajoz, Spain. e-mail address: \url{idelpuerto@unex.es}. ORCID: \href{https://orcid.org/0000-0002-1034-2480}{0000-0002-1034-2480}.}
 }

\begin{document}
\maketitle

\begin{abstract}
The aim of this paper is to introduce a multitype branching process with random migration following the research initiated with the Galton--Watson process with migration introduced in \cite{YM80}. {We} focus our attention in what we call the critical case. Sufficient conditions are provided for the process to have unlimited growth or not. Furthermore, using suitable normalizing sequences, we study the asymptotic distribution of the process. Finally, {we obtain} a Feller-type diffusion approximation.\end{abstract}


\section{Introduction}

The Galton--Watson branching processes with migration have been studied extensively in the literature.  The readers interested in the origin of these models are referred to the surveys \cite{Pa06}, \cite{Ra95} and \cite{VZ93}. These are  stochastic models used in population dynamic and other fields to describe the evolution of populations in which individuals produce offspring according to some probability distribution (often assumed to be independent and identically distributed),   and also have the possibility of migrating. Migration can  involve individuals leaving the population (emigrants) or individuals entering  the population from an external source (immigrants).

In this paper we generalized the  Galton--Watson process with migration introduced in \cite{YM80}. Let $\N$,  $\Z_+$, $\Z$, $\R$ and $\R_+$ be the {sets} of positive integers, non-negative integers, integers, real numbers, and non-negative real numbers, respectively. {Let} 
$\{X_{k,j}: k\in \Z_+, j\in \N\}$ be  a sequence of $\Z_+$-valued   independent and identically distributed (i.i.d.) random variables (offspring variables). On the same probability space, let  
$\{ I_k\}_{k\in\Z_+}$ be defined as a set of $\Z_+$-valued i.i.d. (immigration) random
variables, independent of the offspring variables. A discrete time homogeneous Markov chain $(Y_k)_{k\in \Z_+}$ is called a Galton--Watson process with migration (GWPM) if $Y_0=0$ almost surely and, for $k\in \Z_+$,
$$
Y_{k+1}= \sum_{j=1}^{Y_k} X_{k,j}+M_{k+1}, 
$$
(with $\sum_1^0=0$), where for $p,q,r\geq 0$ such that $p+q+r=1$,
\begin{eqnarray} \label{M+}
M_{k+1}=\left\{
       \begin{array}{lll}
    
    \ \  0                & \mbox{with probability} \ p,  & \quad \mbox{(no migration)},\\
     \ \  I_{k+1} & \mbox{with probability} \ q, &  \quad \mbox{(immigration)},\\
    -X_{k,1}\mathds{1}_{\{Y_k>0\}}  & \mbox{with probability} \ r,  &  \quad \mbox{(emigration)},  
   
        \end{array}
   \right.
\end{eqnarray}
is the migration component. {Here $\mathds{1}_A$ is} the indicator function of a set $A$. As usual  $Y_k$ represents the population size at generation $k$ and $X_{k,j}$ {is} the number of offspring {of} the $j$-th individual at the $k$-th generation. {Notice} that the model considers the possibility that {only a single} family can emigrate.

This model was studied in depth in the papers \cite{Dy97}, \cite{Kh80}, \cite{YM83} and  \cite{YM84c}. Later, the model of GWPM introduced in \cite{YM80} was {extended} in \cite{YY96} and \cite{YY00} {by} allowing two types of emigration (family and individual) as well as state-dependent immigration. 

The aim of this paper is to introduce a more general branching process with migration, by considering different types of individuals in the population, and a more general scheme of immigration {and} emigration components. In Section \ref{sec2}, {we define} the muti-type branching process with random  migration (MBPM), and {obtain} its first and second conditional moments. Focusing our attention in what we call critical case, {we prove} sufficient conditions for the unlimited growth or not of the process in Section \ref{sec3}. Finally, Section \ref{sec4} is dedicated to the study of the limiting behaviour of the process. First, {we consider} several normalizing sequences to obtain different limiting distributions, namely, gamma and normal distributions. Secondly, we consider a sequence of appropriately scaled random step
functions formed from a critical MBPM and prove its weak convergence to a Feller-type diffusion process. {The proof of this result is relegated to the Appendix.}

\section{Probability model}\label{sec2}
Let us fix a probability space $(\Omega, \mathcal{F},\mathrm{P})$ where all the random vectors are defined and let  $p\in\N$ be the  dimension of the vectors. 
For $\bz=(z_1,\ldots,z_p)^\top \in\R^p$, let 
   $\bz^+ =(z_1^+,\ldots,z_p^+)^\top \in\R_+^p$, 
 where  $x^+$ stands for the positive part of $x\in\R$ and  $^\top$ for the transpose.

Let  $\{\boldsymbol{X}_{k,j,i} = ({X}_{k,j,i,1},\ldots,{X}_{k,j,i,p})^\top:i\in\{1,\ldots,p\},\, k\in\Z_+,\,j\in\N\}$ be
 $\Z_+^p$-valued independent random vectors, i.d. for each $i\in\{1,\ldots,p\}$. {Let 
$\{\boldsymbol{D}_{k}(\bz) = (D_{k,1}(\bz),\ldots,D_{k,p}(\bz))^\top : k\in\Z_+\}$ be  $\Z_+^p$--valued  i.d. random vectors for each 
$\bz \in\Z_+^p$ such that $D_{k,i}(\bz)=0$ if $z_i=0$, and it has range on $[1,z_i]$ if $z_i>0$, $i\in\{1,\ldots,p\}$. Let 
$\{\boldsymbol{I}_{k} (\bz) = (I_{k,1}(\bz),\ldots,I_{k,p}(\bz))^\top : k\in\Z_+\}$ be
 also $\N^p$--valued  identically distributed random vectors for each 
$\bz \in\Z_+^p$.} For $\boldsymbol{z}\in\Z_+^p$, let  $\{\boldsymbol{M}_{k}(\bz) = (M_{k,1}(\bz),\ldots,M_{k,p}(\bz))^\top : k\in\Z_+\}$ 
be identically distributed $\Z^p$-valued random vectors  with range contained in 
$[-z_1,\infty) \times \cdots \times [-z_p,\infty)${, defined}, for all $k\in\Z_+$ and $i\in\{1,\ldots,p\}$, {as}
\begin{equation*}
    M_{k,i}(\bz) = 
    \begin{cases}
    \ \  0 & \text{ with probability } p_i(\bz)\in[0,1] ,\\
      \ \  I_{k,i}(\bz) & \text{ with probability } q_i(\bz)\in[0,1] ,\\
        -D_{k,i}(\bz) & \text{ with probability } r_{i}(\bz)\in[0,1] ,\\
            \end{cases}
\end{equation*}
with $p_i(\bz)+q_i(\bz)+r_i(\bz)=1$, and
 $r_i(\bz)$ vanishing if $z_i = 0$. 

Let  $\bZ_0 = (Z_{0,1},\ldots,Z_{0,p})^\top$ be a $\Z_+^p$--valued random vector {(the initial generation)} and let us assume that $\{\boldsymbol{Z}_{0} ,\boldsymbol{X}_{k,j,i} , \boldsymbol{M}_{k}(\bz): k\in\Z_+,\, j\in\N,\, i\in\{1,\ldots,p\},\, \boldsymbol{z}\in\Z_+^p\}$ are
independent $\Z^p$--valued random vectors. 

We define a multitype branching process with random migration (MBPM) as a sequence of  $p$-dimensional random vectors, ${(\boldsymbol{Z}_{k})_{k\in\Z_+}}$,  recursively as 
\begin{equation}\label{mod}
    \boldsymbol{Z}_{k+1}  = 
    \sum_{i=1}^{p} \sum_{j=1}^{Z_{k,i}+M_{k,i}(\bZ_k)} \boldsymbol{X}_{k,j,i}, 
    \quad \quad k\in\Z_+.
\end{equation}
Notice that \eqref{mod} generalizes (\ref{M+}) in several {directions}. Model \eqref{mod} considers that there exist $p$-types of individuals {and} each one can give birth {to} individuals of any type. Moreover, several individuals (not necessary {from} the same family) can leave the population. {The} immigrants  {produce offspring} in the same generation in which they arrive.  

Let us introduce notations for some moments. For $i\in\{1,\ldots,p\}$ and $\bz\in\Z_+^p$,
\begin{align*}
    \boldsymbol{m}_i &=  \ev{\boldsymbol{X}_{0,1,i}}\in\R_+^p, & \mathsf{m}&=( \boldsymbol{m}_1,\ldots,  \boldsymbol{m}_p)\in\R_+^{p\times p}\\
    \mathsf{\Sigma}_i &= \var{\boldsymbol{X}_{0,1,i}}\in\R^{p\times p}, & \boldsymbol{\mathsf{\Sigma}} &= (\mathsf{\Sigma}_1, \ldots,\mathsf{\Sigma}_p)\in (\R^{p\times p})^p\\
    a_i(\bz)&=\ev{I_{0,i}(\bz)}\in\R_+, & \boldsymbol{a}(\bz)&=( a_1(\bz),\ldots, a_p(\bz))^\top\in \R_+^{p}\\
    b_i(\bz)&=\ev{D_{0,i}(\bz)}\in\R_+, & \boldsymbol{b}(\bz)&=( b_1(\bz),\ldots, b_p(\bz))^\top\in \R_+^{p}
\end{align*}
Let us denote
the operators 
\begin{eqnarray*}
 \circ :\R^p \times \R^{p}
 &\to&  \R^p,\\
 (\bz_1,\bz_2)&\to&\bz_1\circ\bz_2=(z_{1,1}z_{2,1},\ldots,z_{1,p}z_{2,p})^\top 
\end{eqnarray*}
 and 
 \begin{eqnarray*}
 \odot :\R^p \times (\R^{p\times p})^p  
 &\to&  \R^{p\times p}\\
 (\boldsymbol{z},\boldsymbol{\mathsf{A}})\hspace*{1.0cm}& \to &\boldsymbol{z}\odot\boldsymbol{\mathsf{A}} 
    =  \sum_{i=1}^p z_i\mathsf{A}_i.
\end{eqnarray*} 
Notice that the operator $\circ$ is the Hadamard product of matrices (see \cite{hj}) applied to column vectors. 

It is easy to see that a MBPM is a homogeneous
multitype Markov chain, whose state space  is $\Z_+^p$. It is interesting to notice that the state  ${\bf 0}$ is absorbing if and if $q_i({\bf 0})=0$ for every $i\in\{1,\ldots,p\}$. In addition, if 
$$\mathrm{P}\left(\bigcap_{i=1}^p\left(\{M_{0,i}(\bz)=-z_i\} \cup \{M_{0,i}(\bz)>-z_i,  \boldsymbol{X}_{0,j,i}= {\bf 0}, j=1, \ldots, z_i+M_{0,i}(\bz)\}\right)\right)>0,$$
for each non-null state $\bz$, then 
$$\mathrm{P}(\boldsymbol{Z}_n= \bz \mbox{ for some } n>0\mid \boldsymbol{Z}_0=\bz)<1,$$
{because $\mathrm{P}(\boldsymbol{Z}_n = \boldsymbol{0}\mid \boldsymbol{Z}_0=\bz)>0$ for all $n>0$}. Consequently,  such state $\bz$ is transient. Under the assumptions that 
 ${\bf 0}$ is an absorbing state and that each non null state is transient, making use of Markov chain theory, we have the extinction--explosion {duality}
\begin{equation}\label{ex} \mathrm{P}(\boldsymbol{Z}_n\to{\bf 0 })+\mathrm{P}(\|\boldsymbol{Z}_n\|\to\infty)=1.\end{equation}

In the following proposition we establish the first and second conditional moments of the process. Let us consider the canonical filtration of the process $\mathcal{F}_k =\sigma(\boldsymbol{Z}_{0},\ldots,\boldsymbol{Z}_{k}),$ $k\in\Z_+$. 
 For $\bz\in \Z_+^p$, let
$$
\boldsymbol{h}(\bz) = \ev{\boldsymbol{M}_0(\bz)}= \boldsymbol{a}(\bz)\circ \boldsymbol{q}(\bz)-\boldsymbol{b}(\bz)\circ \boldsymbol{r}(\bz),
$$
{where $\boldsymbol{q}(\bz)=( q_1(\bz),\ldots, q_p(\bz))^\top$ and $\boldsymbol{r}(\bz)=( r_1(\bz),\ldots, r_p(\bz))^\top$.}

\begin{proposition}\label{p1}
Let ${(\boldsymbol{Z}_{k})_{k\in\Z_+}}$ be a  MBPM. {Then}, for each $k\in\N$, 
      \begin{align}
        \evcond{\boldsymbol{Z}_{k}}{\mathcal{F}_{k-1}} &= \mathsf{m}(\boldsymbol{Z}_{k}+\boldsymbol{h}(\boldsymbol{Z}_{k})), \label{eq-conditional-expected-value} \\
        \var{\boldsymbol{Z}_{k}\mid \mathcal{F}_{k-1}}& =  (\boldsymbol{Z}_{k}+\boldsymbol{h}(\boldsymbol{Z}_{k}))\odot\boldsymbol{\mathsf{\Sigma}}+\mathsf{m} \var{\boldsymbol{{M}}_{0}(\boldsymbol{Z}_{k})}\mathsf{m}^\top,\label{eq-conditional-variance}
    \end{align}
    almost surely.
   \end{proposition}
\begin{proof}
  Using the Markov property and the independence between the random  vectors in the model, we have
 \begin{eqnarray*}
\ev{\boldsymbol{Z}_{k+1}\mid \boldsymbol{Z}_{k}=\bz}&=& \ev{ \sum_{i=1}^{p} \sum_{j=1}^{{z_i}+M_{k,i}(\bz)} \boldsymbol{X}_{k,j,i}}=\ev{\ev{\sum_{i=1}^{p} \sum_{j=1}^{{z_i}+M_{k,i}(\bz)} \boldsymbol{X}_{k,j,i}\mid M_{k,i}(\bz)}}\\&=&\sum_{i=1}^p (z_{i}+\ev{M_{k,i}(\bz)})\boldsymbol{m}_i= \mathsf{m} (\bz+\ev{\boldsymbol{M}_0(\bz)})\\&=&\mathsf{m}(\bz+\boldsymbol{h}(\bz)),\end{eqnarray*}
\begin{eqnarray*}
 \var{\boldsymbol{Z}_{k+1}\mid \boldsymbol{Z}_{k}=\bz}&=& \ev{\var{\sum_{i=1}^{p} \sum_{j=1}^{{z_i}+M_{k,i}(\bz)} \boldsymbol{X}_{k,j,i}\mid M_{k,i}(\bz)}}\\&&+\var{\ev{\sum_{i=1}^{p} \sum_{j=1}^{{z_i}+M_{k,i}(\bz)} \boldsymbol{X}_{k,j,i}\mid M_{k,i}(\bz)}}\\&=&\sum_{i=1}^p
(z_{i}+\ev{M_{k,i}(\bz)})\mathsf{\Sigma}_i+\var{\mathsf{m}\boldsymbol{M}_{k}(\bz)}\\ &=& (\bz+\ev{\boldsymbol{M}_0(\bz)})\odot\boldsymbol{\mathsf{\Sigma}}+\mathsf{m} \var{\boldsymbol{{M}}_{0}(\bz)}\mathsf{m}^\top\\
& = & (\bz+\boldsymbol{h}(\bz))\odot\boldsymbol{\mathsf{\Sigma}}+\mathsf{m} \var{\boldsymbol{{M}}_{0}(\bz)}\mathsf{m}^\top.
\end{eqnarray*}
    
\end{proof}

\section{Unlimited growth of the process}\label{sec3}
 In this section we focus our attention in studying $\mathrm{P}(\|\boldsymbol{Z}_n\|\to \infty)$. {Here $\boldsymbol{Z}_n$ is} what we call a \textit{critical} MBPM. {It satisfies} the following hypothesis:
\begin{quote}
{\bf Hypothesis A:} The offspring mean matrix $\mathsf{m}$ is primitive with Perron--Frobenius eigenvalue equals 1.
\end{quote}

The Perron--Frobenius' theorem (see \cite{hj}) guarantees that there exist left and right eigenvectors, $\boldsymbol{u}=(u_1\ldots,u_p)^\top$ and $\boldsymbol{v}=(v_1,\ldots,v_p)^\top$, respectively, of the Perron-Frobenius eigenvalue such that all their coordinates are {positive},  $\sum_{i=1}^p u_i=1$ and $\boldsymbol{u}^\top \boldsymbol{v} =1$.

{We} also consider the following hypothesis {for} the expected value of the random migration component:
\begin{quote}
{\bf Hypothesis B:} The {vector} $\boldsymbol{h}(\bz)=(h_1(\bz),\ldots, h_p(\bz))^\top$, $\bz\in \Z_+^p$, satisfies $h_i(\bz)=\oo(\|\bz\|)$ as $\|\bz\| \to \infty$, for every $i\in\{1,\ldots,p\}$.
\end{quote}
It is interesting to notice that under Hypotheses A and B the mean growth rates of the functional $\boldsymbol{u}^\top\boldsymbol{Z}_{n}$ converge to 1. Indeed, as $\|\bz\| \to \infty$,
$$\frac{1}{\boldsymbol{u}^\top\bz} \ev{\boldsymbol{u}^\top\boldsymbol{Z}_{n+1} \mid \boldsymbol{Z}_n=\bz} = \frac{1}{\boldsymbol{u}^\top\bz} {\boldsymbol{u}^\top \mathsf{m} (\bz + \boldsymbol{h}(\bz))} = \frac{1}{\boldsymbol{u}^\top\bz} {\boldsymbol{u}^\top (\bz + \boldsymbol{h}(\bz))} = 1 + \frac{\boldsymbol{u}^\top \boldsymbol{h}(\bz)}{\boldsymbol{u}^\top\bz} \to 1.$$

Let us define $\sigma^2(\bz)\coloneqq \var{\boldsymbol{u}^\top\boldsymbol{Z}_{n+1}\mid \boldsymbol{Z}_{n}=\bz}$, $\bz\in \Z_+^p$. Taking into account Hypothesis A and Proposition \ref{p1}, it is easy to obtain that
$$
\sigma^2(\bz)=  \boldsymbol{u}^\top\left((\bz+\boldsymbol{h}(\bz))\odot\boldsymbol{\mathsf{\Sigma}}+ \var{\boldsymbol{M}_{0}(\bz)}\right)\boldsymbol{u}.
$$
{Next} we establish sufficient conditions {under which} the process {does not} have unlimited growth.

\begin{theorem}\label{explocero}
Let ${(\boldsymbol{Z}_{k})_{k\in\Z_+}}$ be a  MBPM satisfying Hypotheses A and B. Then $\mathrm{P}(\|\boldsymbol{Z}_{n}\|\to {\infty})=0$  if at least one of the following conditions holds:

\begin{itemize}
 \item[(a)] For each $i\in\{1,\ldots, p\},$  $h_i(\bz)\leq 0$ for all $\bz\in\Z_+^p$ with $\|\bz\|$ large enough.
    \item[(b)] \begin{equation}\label{limsup}\limsup_{\|\bz\|\to\infty}\frac{2({\boldsymbol{u}^\top\bz} )(\boldsymbol{u}^\top\boldsymbol{h}(\bz)) }{\sigma^{2}(\bz)}<1,\end{equation}
     and, for every $i\in\{1,\ldots, p\}$,
     \begin{equation}\label{cond1}
         \ev{|z_i+M_{0,i}(\bz)|^{1+\delta/2}}=\oo((\boldsymbol{u}^\top\bz )^{1+\delta} \sigma^2(\bz)),
     \end{equation}
  and 
\begin{equation}\label{cond2}
         \ev{|M_{0,i}(\bz)-h_i(\bz)|^{2+\delta}}=\oo((\boldsymbol{u}^\top\bz )^{1+\delta} \sigma^2(\bz)),
     \end{equation}
  as $\|\bz\|\to \infty$, for some $0<\delta\leq 1$.
\end{itemize}
  \end{theorem}
\begin{proof} If \textit{(a)} is satisfied, Theorem 3 in \cite{gmm} {implies} that $\mathrm{P}(\|\boldsymbol{Z}_n\|\to\infty)=0$. If \textit{(b)} holds, the result is a consequence of Theorem 4 in \cite{gmm}, {provided} that, as $\|\bz\|\to \infty$,
\begin{equation}\label{cond0}
\xi(\bz)\coloneqq    \ev{\left|\boldsymbol{u}^\top\boldsymbol{Z}_{n+1} -\ev{\boldsymbol{u}^\top\boldsymbol{Z}_{n+1} \mid \boldsymbol{Z}_n}\right|^{2+\delta} \mid \boldsymbol{Z}_n=\bz}=\oo((\boldsymbol{u}^\top\bz)^{1+\delta} \sigma^2(\bz)).
\end{equation}
Indeed,    using the fact that for every $a,b\in\R$ and $r>0$ there exists a positive constant $C_r$ such that $|a+b|^r\leq
C_r(|a|^r+|b|^r)$ ($C_r$-inequality) and the Marcinkiewicz-Zygmund's inequality (see Lemma \ref{l3} in Appendix), it is easy to check that 
\begin{eqnarray}
  \xi(\bz)&\leq& C_1 \sum_{i=1}^p \left(\ev{\left |\sum_{j=1}^{z_i+{M}_{0,i}(\bz)}\boldsymbol{u}^\top(\boldsymbol{X_{0,j,i}}-\boldsymbol{m}_i)\right|^{2+\delta}}+\ev{|M_{0,i}(\bz)\boldsymbol{u}^\top\boldsymbol{m}_i-h_i(\bz) \boldsymbol{u}^\top\boldsymbol{m_i}|^{2+\delta}}\right)\nonumber\\
  &\leq &  C_2  \sum_{i=1}^p \left(\ev{(z_i+M_{0,i}(\bz))^{1+\delta/2}} + \ev{|M_{0,i}(\bz)-h_i(\bz)|^{2+\delta}}(\boldsymbol{u}^\top\boldsymbol{m}_i)^{2+\delta} \right),\label{eq7}
\end{eqnarray}
with $C_1$ and $C_2$ certain positive constants. Hence, using \eqref{cond1} and \eqref{cond2}, we obtain \eqref{cond0}.
\end{proof}

\begin{remark}
\begin{enumerate}
\item[]
\item[i)] In hypothesis \textit{(a)} we {assume} that for each type the emigration {dominates} over the immigration {given} large enough population size. 
\item[ii)] To establish {Theorem \ref{explocero}} is not necessary to {require} that the extinction-explosion duality   \eqref{ex} holds. {It} is also valid for migration models with immigration at $\bf 0$, that is, with a reflecting barrier at $\bf 0$. {In the case in which $\bf 0$ is an absorbing state} and the non-null states are transient, the result provides sufficient conditions for the extinction of the population.
\end{enumerate}
\end{remark}

{Next} we study sufficient  conditions {for} a positive probability of unlimited growth of the process. 

  \begin{theorem} \label{teo2}Let ${(\boldsymbol{Z}_{k})_{k\in\Z_+}}$ be a  MBPM  satisfying Hypotheses A and B and such that, {for every non-null $\bz \in \Z_+$,}  
  \begin{equation}\label{eqex}
  \mathrm{P}\left(I_{0,i}(\bz)>0, {\mathbf{X}_{0,j,i}}\not ={\bf 0}, i\in A(\bz), j=1,\ldots z_i+I_{0,i}(\bz)  \right)>0,
  \end{equation}
  where $A(\bz)=\{i\in\{1,\ldots, p\}: z_i>0\}$. Then $\mathrm{P}\left(\|\boldsymbol{Z}_n\|\to \infty\right)>0$ if 
\begin{equation}\label{liminferior}\liminf_{\|\bz\|\to\infty}\frac{2({\boldsymbol{u}^\top\bz} )(\boldsymbol{u}^\top\boldsymbol{h}(\bz)) }{\sigma^{2}(\bz)}>1,
\end{equation}
and, for every $i\in\{1,\ldots, p\}$,
     \begin{equation}\label{cond3}
         \ev{|z_i+M_{0,i}(\bz)|^{1+\delta/2}}=\oo((\boldsymbol{u}^\top\bz )^{1+\delta} (\boldsymbol{u}^\top\boldsymbol{h}(\bz))/(\log(\boldsymbol{u}^\top\bz ))^{1+\alpha})
     \end{equation}
     and 
      \begin{equation}\label{cond4}
          \ev{|M_{0,i}(\bz)-h_i(\bz)|^{2+\delta}}=\oo((\boldsymbol{u}^\top\bz )^{1+\delta} (\boldsymbol{u}^\top\boldsymbol{h}(\bz))/(\log(\boldsymbol{u}^\top\bz ))^{1+\alpha})
     \end{equation}
     as $\|\bz\|\to \infty$, for some $0<\delta\leq 1$ and $\alpha>0$.
  \end{theorem}
\begin{proof}
Taking into account  Theorem 2 in \cite{gmm} {and} \eqref{eqex} {holds}, it is enough to {show} that
$\xi(\bz)=\oo((\boldsymbol{u}^\top\bz )^{1+\delta} (\boldsymbol{u}^\top\boldsymbol{h}(\bz))/(\log(\boldsymbol{u}^\top\bz ))^{1+\alpha})$. This later follows from (\ref{eq7}), \eqref{cond3} and \eqref{cond4}.
\end{proof}

\begin{remark} Theorem \ref{explocero} also holds if \eqref{cond1} and \eqref{cond2} are replaced by
 \begin{equation}\label{cond1'}
         \ev{|z_i+M_{0,i}(\bz)|^{1+\delta/2}}=\oo((\boldsymbol{u}^\top\bz )^{2+\delta} \boldsymbol{u}^\top\boldsymbol{h}(\bz))
     \end{equation}
  and 
\begin{equation}\label{cond2'}
         \ev{|M_{0,i}(\bz)-h_i(\bz)|^{2+\delta}}=\oo((\boldsymbol{u}^\top\bz )^{2+\delta} \boldsymbol{u}^\top\boldsymbol{h}(\bz)),
     \end{equation}
    respectively.
Indeed, (\ref{eq7}), \eqref{cond1'} and \eqref{cond2'} {yield} $\xi(\bz)=\oo((\boldsymbol{u}^\top\bz )^{2+\delta} \boldsymbol{u}^\top\boldsymbol{h}(\bz))$. {Therefore},
$$\frac{\xi(\bz)}{\sigma^2(\bz)(\boldsymbol{u}^\top\bz)^{1+\delta}}= \frac{\xi(\bz)}{(\boldsymbol{u}^\top\bz)^{2+\delta} (\boldsymbol{u}^\top\boldsymbol{h}(\bz))}\frac{(\boldsymbol{u}^\top\bz) (\boldsymbol{u}^\top\boldsymbol{h}(\bz))}{\sigma^2(\bz)}.$$
Hence, by assumption (\ref{limsup})  and $\xi(\bz)=\oo((\boldsymbol{u}^\top\bz )^{2+\delta} \boldsymbol{u}^\top\boldsymbol{h}(\bz))$, we deduce   $\xi(\bz)=\oo((\boldsymbol{u}^\top\bz)^{1+\delta} \sigma^2(\bz)).$ 

Analogously, in Theorem \ref{teo2} we can replace (\ref{cond3}) and (\ref{cond4}) with
 \begin{equation*}
         \ev{|z_i+M_{0,i}(\bz)|^{1+\delta/2}}=\oo(\sigma^2(\bz)(\boldsymbol{u}^\top\bz)^{\delta}/(\log(\boldsymbol{u}^\top\bz))^{1+\alpha}),
     \end{equation*}
     and 
      \begin{equation*}
          \ev{|M_{0,i}(\bz)-h_i(\bz)|^{2+\delta}}=\oo(\sigma^2(\bz)(\boldsymbol{u}^\top\bz)^{\delta}/(\log(\boldsymbol{u}^\top\bz))^{1+\alpha}),
     \end{equation*}
respectively.
Assumptions \eqref{limsup} and  \eqref{liminferior}
 could be intuitively interpreted for a critical MBPM as speed {conditions},  noticing that
$$\frac{(\boldsymbol{u}^\top\bz ) (\boldsymbol{u}^\top\boldsymbol{h}(\bz))}{\sigma^2(\bz)}=\frac{(\boldsymbol{u}^\top\bz)^{-1}\ev{\boldsymbol{u}^\top\boldsymbol{Z}_{n+1}\mid \boldsymbol{Z}_n=\bz}-1}{\var{\boldsymbol{u}^\top\boldsymbol{Z}_{n+1}\mid \boldsymbol{Z}_{n}=\bz} (\boldsymbol{u}^\top\bz)^{-2}}.$$

\end{remark}

\section{Limiting behaviour of the process}\label{sec4}
Consider a positive and twice continuously differentiable real function $\overline{h}(x)$ such that $\boldsymbol{u}^\top \boldsymbol{h}(\bz)=\overline{h}(\boldsymbol{u}^\top \bz )$, for $\bz\in \Z_+^p$.  We denote by $\{a_n\}_{n\in \Z_+}$ the solution of the {difference} equation
$$a_0=1,\qquad a_{n+1}=a_n+\overline{h}(a_n).$$ 

In the next result, we consider different normalizing sequence of constants that guarantee the convergence of the corresponding normalized process on the unlimited growth set.


\begin{theorem}\label{t3} Let ${(\boldsymbol{Z}_{k})_{k\in\Z_+}}$ be a  MBPM  satisfying Hypotheses A and B,  and (\ref{eqex}). Suppose that for every non-null vector $\bz\in \Z_+^p$:
\begin{enumerate}
\item[(i)] $h_i(\bz)= c_i(\boldsymbol{u}^\top\bz)^{\alpha}+\oo\left((\boldsymbol{u}^\top\bz)^{\alpha}\right)$ for each $i=1,\ldots, p$ and for some $\alpha<1$ and $\boldsymbol{c}=(c_1,\ldots,c_p)^\top\in \R^p$ such that $\boldsymbol{u}^\top \boldsymbol{c} > 0$.
\item[(ii)] $\sigma^2(\bz)=\nu(\boldsymbol{u}^\top\bz)^\beta+\oo\left((\boldsymbol{u}^\top\bz)^\beta\right)$ for some $\beta\leq 1+\alpha$ and $\nu>0$.
\item[(iii)] $\max_{1\leq i\leq p} \left\{\ev{|z_i+M_{0,i}(\bz)|^{1+\delta/2}}, \ev{|M_{0,i}(\bz)-h_i(\bz)|^{2+\delta}} \right\}=\OO\left((\sigma^2(\bz))^{1+\delta/2}\right)$ for some $\delta\in(0,1]$.
\item[(iv)] $\max_{1\leq i\leq p}\{|h_i(\bz)|\}=\OO\left((\boldsymbol{u}^\top\bz)^{\delta_1}\right)$ for some $\delta_1<1$.
\item[(v)]  For some $\tilde{\alpha}$, $1<\tilde{\alpha}\leq 2$, $\max_{1\leq i\leq p}\left\{z_i+h_{i}(\bz), \ev{|M_{0,i}(\bz)-h_i(\bz)|^{\tilde{\alpha}}} \right\}=\OO\left((\boldsymbol{u}^\top\bz)^{\delta_2\tilde{\alpha}}\right)$ for some $\delta_2<1$.
\end{enumerate}
Then \begin{enumerate}
\item[a)] If $\beta=1+\alpha$ and $\nu<2(\boldsymbol{u}^\top\boldsymbol{c})$, then for every vector
 ${\boldsymbol{x}}
\in\mathbb{R}^p$
$$\lim_{n\to\infty}\mathrm{P}\left(\frac{\boldsymbol{Z}_n}{n^{1/(1-\alpha)}}\leq
{\boldsymbol{x}} \mid \|\boldsymbol{Z}_n\|\to\infty\right)=F_{\boldsymbol{u}Z}\left({\boldsymbol{x}}\right).$$ \noindent
{Here} $F_{\boldsymbol{u}Z}$ {is} the distribution function associated with the random vector $\boldsymbol{u}Z$, and ${Z}$ {is} a random variable such that ${Z}^{1-\alpha}$ follows a gamma distribution with parameters $\frac{2(\boldsymbol{u}^\top\boldsymbol{c})-\nu\alpha}{\nu(1-\alpha)}$
and  $\frac{\nu(1-\alpha)^2}{2}$. {Note that} vector inequalities are evaluated component-wise.

\item[b)] If $0<\alpha<1$ and $\beta<\alpha +1$, then, on
$\{\| \boldsymbol{Z}_n\|\to\infty\}$, $n^{1/(\alpha-1)}\boldsymbol{Z}_n$ converges to $((\boldsymbol{u}^\top \boldsymbol{c})(1-\alpha))^{1/(1-\alpha)}\boldsymbol{u}$ in $L^1$.

Moreover, if $\beta\geq 3\alpha-1$ and
$\max\{\delta_1,\delta_2\}<(\beta-\alpha+1)/2$, then for every vector $\boldsymbol{x}\in\mathbb{R}^p$, it is verified that
$$\lim_{n\to\infty}\mathrm{P}\left(\frac{\boldsymbol{Z}_n-\boldsymbol{u}a_n}{\Lambda_n}\leq {\boldsymbol{x}}
\mid \| \boldsymbol{Z}_n\|\to\infty\right)=F_{\boldsymbol{u}U}\left({\boldsymbol{x}}\right).$$ \noindent {Here} $F_{\boldsymbol{u}U}\left({\boldsymbol{x}}\right)$ {is} the distribution function associated with the random vector $\boldsymbol{u}U$ {and} $U$ {is}  a random variable with standard normal distribution. {The normalization sequence is given by}  
$$\Lambda_n:=\begin{cases}\nu^{1/2}((\boldsymbol{u}^\top\boldsymbol{c})(1-\alpha))^{(3\alpha-1)/(2(1-\alpha))}
n^{\alpha/(1-\alpha)}(\log n)^{1/2} & \mbox{ if }
\beta=3\alpha-1,\cr (\nu(\beta-3\alpha+1)^{-1}(\boldsymbol{u}^\top\boldsymbol{c})^{\beta/(1-\alpha)}
((1-\alpha)n)^{(\beta-\alpha+1)/(1-\alpha)})^{1/2} & \mbox{ if }
\beta>3\alpha-1.
\end{cases}$$
\end{enumerate}
\end{theorem}
\begin{proof}
The proof follows from Corollary 4.1 and Remark 4.1 in \cite{gmm06}, {observing that} the MBPM {is} a multitype controlled branching process (MCBP) with $\phi_k(\bz)=\bz+M_k(\bz)$, $k\in \Z_+$ and $\bz\in \Z_+^p$.
\end{proof}

\begin{remark}\label{r2}
\begin{enumerate}
\item[]
\item[i)] Notice that under {the} assumptions of Theorem \ref{t3}, Theorem
\ref{teo2} holds, and therefore $\mathrm{P}(\| \boldsymbol{Z}_n\|\to\infty)>0$. In fact, from (i) and (ii), it is
derived that
$$
\label{liminf}\liminf_{\|z\|\to\infty}\frac{2({\boldsymbol{u}^\top\bz} )(\boldsymbol{u}^\top\boldsymbol{h}(\bz)) }{\sigma^{2}(\bz)}= 2(\boldsymbol{u}^\top\boldsymbol{c})\nu^{-1} \quad
\mbox{ if }\quad \beta=\alpha+1 \quad \mbox{ or } \quad \infty
\quad \mbox{ otherwise.}
$$ Finally,  using again (ii) and (iii), we have that
$\xi(\bz)=\oo((\boldsymbol{u}^\top \bz )^{1+\delta} (\boldsymbol{u}^\top\boldsymbol{h}(\bz))/(\log(\boldsymbol{u}^\top \bz ))^{1+\alpha})$. Note that,
{if} $\beta>\alpha +1$ or {if} $\beta=\alpha+1$ and $2(\boldsymbol{u}^\top\boldsymbol{c})\nu^{-1}<1$, {then}
$\mathrm{P}(\| \boldsymbol{Z}_n\|\to\infty)=0$ (see Theorem
\ref{explocero}).
\item[ii)] It
is not hard to {see} that $a_n \sim ((\boldsymbol{u}^\top \boldsymbol{c})(1-\alpha)n)^{(1-\alpha)^{-1}},$ as
$n\to\infty$.
\end{enumerate}
\end{remark}

\subsection{Scaling limits}
We now address the problem of scaling limits of MBPMs. To {this end}, we consider our process as particular case of a MCBP and {apply} the results on scaling limits for MCBPs provided in  \cite{bgmp}.

{Assume} the following hypothesis {holds}.
\begin{quote}
{\bf Hypothesis C:} There exist vectors $\boldsymbol{a},\boldsymbol{b} \in \R_+^p$ and $\boldsymbol{q},\boldsymbol{r} \in [0,1]^p,$ such that, as $\|\bz\|\to \infty$,
$$\boldsymbol{a}(\bz)\to \boldsymbol{a},\quad \boldsymbol{b}(\bz)\to \boldsymbol{b}, \quad \boldsymbol{q}(\bz)\to \boldsymbol{q} \quad\text{and}\quad \boldsymbol{r}(\bz)\to \boldsymbol{r}.$$
\end{quote}

Notice that under Hypothesis C, we have that, as $\|\bz\|\to \infty$,  $$\boldsymbol{h}(\bz)\to \boldsymbol{a}\circ \boldsymbol{q}-\boldsymbol{b}\circ \boldsymbol{r}.$$

Notice also that Hypothesis C implies Hypothesis B. Therefore, under Hypothesis C (replacing Hypothesis B) and the rest of assumptions of Theorem \ref{teo2}, we {have} that the process has a positive probability of unlimited growth. {We will assume that $P(\|\boldsymbol{Z}_n\|\to\infty)=1$.}

 {\textbf{}\textit{}\textbf{}{Furthermore, we suppose that in each generation the immigration and emigration components are independent, that is the random variables $\{I_{k,i}(\bz),D_{k,j}(\bz):i\in\{1,\ldots,p\},\,j\in\{1,\ldots,p\},\,\bz\in\Z_+^p\}$ are independent, for every $k \in \Z_+$.}}
 
 {Denote} by $\stackrel{\mathcal{L}}{\longrightarrow}$ the weak convergence on the space of $\R^p$--valued càdlàg functions on $\R_+$ endowed with the Skorokhod metric.
\begin{theorem}\label{Th4}
Let ${(\boldsymbol{Z}_{k})_{k\in\Z_+}}$ be a MBPM  satisfying Hypothesis A, Hypothesis C with $\boldsymbol{a}\circ \boldsymbol{q}-\boldsymbol{b}\circ \boldsymbol{r} {\in\R_+^p\setminus\{\boldsymbol{0}\}}$, and the {following} assumptions:
\begin{enumerate}
    \item[(i)] $\ev{\|\bZ_0\|^2}$,  $\ev{\|\boldsymbol{I}_{0}(\bz)\|^4}$ and $\ev{\|\boldsymbol{X}_{0,1,i}\|^4}$ are finite for each $i=1,\ldots,p$, and $\bz \in \Z_+^p$,
    \item[(ii)] $\ev{D_{0,i}(\bz)^2}=\oo(\|\boldsymbol{z}\|)$ and $\ev{I_{0,i}(\bz)^2} =\oo(\|\boldsymbol{z}\|)$, as $\|\boldsymbol{z}\|\to\infty $ for $i=1,\ldots,p$,
    \item[(iii)] $\ev{D_{0,i}(\bz)^4}=\OO(\|\boldsymbol{z}\|^2)$ and $\ev{I_{0,i}(\bz)^4}=\OO(\|\boldsymbol{z}\|^2)$, as $\|\boldsymbol{z}\|\to\infty $ for $i=1,\ldots,p$.
\end{enumerate}
Then 
\begin{equation*}
    (n^{-1}\bZ_{\lfloor nt \rfloor})_{t\in\R_+} 
    \stackrel{\mathcal{L}}{\longrightarrow} 
    (\mathcal{Z}_t \boldsymbol{v})_{t\in\R_+}
    \qquad \text{as } n\to\infty,
\end{equation*}
     where $(\mathcal{Z}_t)_{t\in\R_+}$ is the pathwise unique strong solution of the SDE
\begin{equation*}
    \dif\mathcal{Z}_t = \boldsymbol{u}^\top \boldsymbol{\alpha}\dif t + \sqrt{\boldsymbol{u}^\top (\boldsymbol{v}\odot\boldsymbol{\mathsf{\Sigma}})\boldsymbol{u}\mathcal{Z}_t^+}\dif\mathcal{W}_t, \qquad t\in\R_+,
\end{equation*}
with initial value $\mathcal{Z}_0 = 0$,
$\boldsymbol{\alpha} = \boldsymbol{a}\circ \boldsymbol{q}-\boldsymbol{b}\circ \boldsymbol{r}$,  
 $\boldsymbol{\mathsf{\Sigma}} = (\Sigma_1,\ldots, \Sigma_p)$ 
    and $(\mathcal{W}_t)_{t\in\R_+}$ a standard Wiener process.

\end{theorem}

\begin{remark}

\begin{enumerate}
\item[]
\item {If, besides the assumptions given in the definition of the model, we consider that, for each $i \in \{1,\ldots,p\}$, $X_{k,j,i,1}, \ldots,X_{k,j,i,p}$ are independent random variables ($k \in \Z_+$, $j \in \N$),
 then the variance matrices are diagonal, i.e.} $\mathsf{\Sigma}_{i}=\mathrm{diag}(\var{X_{0,1,i,1}},\ldots,\var{X_{0,1,i,p}})$, and it is easy to obtain that $\boldsymbol{u}^\top (\boldsymbol{v}\odot\boldsymbol{\mathsf{\Sigma}})\boldsymbol{u}=\sum_{j=1}^p u_j^2\sum_{i=1}^p v_i \var{\boldsymbol{X}_{0,1,i,j}}$.
\item Let us {give one example} of emigration distributions {which satisfy} assumptions (ii) and (iii). For all $k\in\Z_+$, $i=1,\ldots,p$ and $\boldsymbol{z}\in\Z_+^p$, let $D_{k,i}(\bz)$ be 0 if $z_i=0$, otherwise, we define the distribution of $D_{k,i}(\bz)$ as $\mathrm{P}(D_{k,i}(\bz)=j) \propto \frac{1}{j^3}$, $j=1,\ldots,z_i$. It is easy to obtain that $\ev{D_{k,i}(\bz)^2}=\oo(z_i)$ and $\ev{D_{k,i}(\bz)^4}=\OO(z_i^2)$.
\end{enumerate}
\end{remark}
\section*{Appendix}
\begin{lemma}(Marcinkiewicz-Zygmund's inequality)\label{l3}
If  $\{X_n\}_{n\geq 1} $ is a  sequence of  independent and
identically distributed random variables with $E[X_1]=0$ and
$E[|X_1|^p]<\infty$, $p\geq 2$, then
$E[|\sum_{j=1}^nX_j|^p]=O(n^{p/2})$.
\end{lemma}

We refer the reader    to  \cite{chow}, p. 387 for the proof.

\subsection*{Proof of Theorem \ref{Th4}} We rewrite the MBPM as a controlled multitype branching process (see \cite{gmm}), with $\boldsymbol{\phi}_{k} (\boldsymbol{z})  =  \boldsymbol{z} + \boldsymbol{M}_{k}(\bz)$ for  $k\in\Z_+$ and $\bz\in\Z_+^p$, and apply Theorem 3.3 from \cite{bgmp}. Therefore, we only need to check Hypotheses 1--6 of {that} theorem. 

Hypothesis 1 follows trivially from (i).

For each 
$\boldsymbol{z}\in\Z_+^p$, $
   \boldsymbol{h}(\bz) = \ev{\boldsymbol{M}_0(\bz)}= \boldsymbol{a}(\bz)\circ \boldsymbol{q}(\bz)-\boldsymbol{b}(\bz)\circ \boldsymbol{r}(\bz)$, 
so $\boldsymbol{\varepsilon}(\boldsymbol{z}) = \ev{\boldsymbol{\phi}_{0} (\boldsymbol{z})} = \mathsf{\Lambda}\boldsymbol{z} + \boldsymbol{\alpha} + \boldsymbol{g}(\boldsymbol{z})$ {where} $\mathsf{\Lambda}$ {is} the identity matrix of order $p$, $\boldsymbol{\alpha} = \boldsymbol{a}\circ \boldsymbol{q}-\boldsymbol{b}\circ \boldsymbol{r}$ and $\boldsymbol{g}(\bz) = \boldsymbol{h}(\bz) -(\boldsymbol{a}\circ \boldsymbol{q}-\boldsymbol{b}\circ \boldsymbol{r})$.
Consequently, Hypothesis C implies that $\|\boldsymbol{g}(\bz)\|=\oo(1)$ as $\|\bz\| \to \infty$ and, therefore, Hypothesis 2 in \cite{bgmp} {holds}.

To verify Hypothesis 3, it is enough to show that $\var{\phi_{0,i}(\bz)} = \oo(\|\bz\|)$ as $\|\bz\|\to\infty$ for all $i=1,\ldots,p$. Indeed, taking into account the independence of the immigration and emigration components, {we have}
\begin{align*}
    \var{\phi_{0,i}(\bz)} &= \ev{M_{0,i}(\bz)^2} - \ev{M_{0,i}(\bz)}^2 \\
    &= q_i(\bz) \ev{I_{0,i}(\bz)^2} + r_i(\bz)\ev{D_{0,i}(\bz)^2} - \left(q_i(\bz) \ev{I_{0,i}(\bz)} - r_i(\bz)\ev{D_{0,i}(\bz)}\right)^2,
\end{align*}
and, by Hypothesis C and (ii), we {obtain} Hypothesis 3.

Lyapunov's inequality {yields}  $\ev{|X|^3}^{1/3} \leq \ev{|X|^4}^{1/4}$. 
{Furthermore,}
\begin{align*}
    \kappa_i(\bz) &= \ev{(\phi_{0,i}(\boldsymbol{z}) - \ev{\phi_{0,i}(\boldsymbol{z})})^4} \\
    &= \ev{M_{0,i}(\boldsymbol{z})^4} - 3\ev{M_{0,i}(\boldsymbol{z})}^4 - 4\ev{M_{0,i}(\boldsymbol{z})}\ev{M_{0,i}(\boldsymbol{z})^3} + 6 \ev{M_{0,i}(\boldsymbol{z})}^2\ev{M_{0,i}(\boldsymbol{z})^2}  \\
    &= q_i(\bz) \ev{I_{0,i}(\bz)^4} + r_i(\bz)\ev{D_{0,i}(\bz)^4} - 3\left(q_i(\bz) \ev{I_{0,i}(\bz)} - r_i(\bz)\ev{D_{0,i}(\bz)}\right)^4\\
    &\quad -4 \left(q_i(\bz) \ev{I_{0,i}(\bz)} - r_i(\bz)\ev{D_{0,i}(\bz)}\right)\left(q_i(\bz) \ev{I_{0,i}(\bz)^3} - r_i(\bz)\ev{D_{0,i}(\bz)^3}\right) \\
    &\quad + 6 \left(q_i(\bz) \ev{I_{0,i}(\bz)} - r_i(\bz)\ev{D_{0,i}(\bz)}\right)^2 \left(q_i(\bz) \ev{I_{0,i}(\bz)^2} + r_i(\bz)\ev{D_{0,i}(\bz)^2}\right).
\end{align*}
{Hence,} Hypothesis 4 in \cite{bgmp} holds as a consequence of Hypothesis C, (ii) and (iii).

Hypothesis A {yields} Hypothesis 5 because $\mathsf{\Lambda}$ is the identity matrix which implies $\Tilde{\mathsf{m}} = \mathsf{m}\mathsf{\Lambda} = \mathsf{m}$.

Finally, Hypothesis 6 {follows from the fact that we are assuming that $P(\|\boldsymbol{Z}_n\|\to\infty)=1$ and (iv) of Remark 3.1 in \cite{bgmp}}. Hence, the result follows from the aforementioned Theorem 3.3 in \cite{bgmp}.
\proofend

\section*{Acknowledgements}
We would like to thank the referee for her/his comments that helped us improve the paper.

\section*{Funding}
The authors  are supported by grant PID2019-108211GB-I00 funded by MICIU/AEI/10.13039/
501100011033. P. Mart\'in-Ch\'avez is also grateful to the Ministerio de Ciencias, Innovación y Universidades for support from a predoctoral fellowship Grant No.\ FPU20/06588.
\bibliographystyle{plain}
\bibliography{ref}

\end{document}